\documentclass[11pt]{amsart}
\usepackage{a4wide}
\usepackage[pagewise]{lineno}
\usepackage{mathabx}

\usepackage{amscd,amsmath,amssymb,fancyhdr,color}
\usepackage{scalerel, mathtools}
\usepackage[backref=page]{hyperref}
\renewcommand*{\backref}[1]{}
\renewcommand*{\backrefalt}[4]{%
    \ifcase #1 (Not cited.)%
    \or        (Cited on page~#2.)%
    \else      (Cited on pages~#2.)%
    \fi}
    \hypersetup{
     colorlinks   = true,
     citecolor    = magenta
}
\usepackage{relsize}

\renewcommand{\phi}{\varphi}

\newcommand{\Sp}[1]{\mathrm{Sp}{(#1)}}

\numberwithin{equation}{section}

\def\eqref#1{(\ref{#1})}

\newcommand{\C}{{\mathbb C}}
\newcommand{\R}{{\mathbb R}}

\def\1{\sqrt{-1}\:}
\newcommand{\restrict}[1]{{\left|_{{\phantom{|}\!\!}_{#1}}\right.}}
\newcommand{\cntrct}                
{\hspace{2pt}\raisebox{1pt}{\text{$\lrcorner$}}\hspace{2pt}}

\newcommand{\cad}{\mathcal{D}}

\renewcommand{\bar}{\overline}
\renewcommand{\phi}{\varphi}
\renewcommand{\epsilon}{\varepsilon}

\renewcommand{\dim}{\operatorname{dim}}

\newcommand{\eg}{{\em e.g. }}

\newcounter{Mycounter}[section]
\newcounter{lemma}[section]
\newcounter{claim}[section]

\newcounter{sublemma}[section]

\newcounter{corollary}[section]

\newcounter{theorem}[section]

\newcounter{conjecture}[section]

\newcounter{proposition}[section]

\newcounter{definition}[section]

\newcounter{example}[section]

\newcounter{remark}[section]

\newcounter{problem}[section]

\newcounter{question}[section]

\makeatletter

\@addtoreset{equation}{section}

\@addtoreset{footnote}{section}

\makeatother

\title[A characterization of compact LCHK manifolds]{A characterization  of compact locally conformally hyperk\" ahler manifolds}
\author{Liviu Ornea}
\address[Liviu Ornea]{University of Bucharest, Faculty of Mathematics and Informatics, 14 Academiei Str., Bucharest, Romania, AND\newline
Institute of Mathematics ``Simion Stoilow" of the Romanian Academy, 21,
Calea Grivitei Street, 010702, Bucharest, Romania}
\author{Alexandra Otiman} 
\thanks{Both authors are partially supported by a grant of Ministry of Research and Innovation, CNCS - UEFISCDI,
project number PN-III-P4-ID-PCE-2016-0065, within PNCDI III}

\address[Alexandra Otiman]{
Institute of Mathematics ``Simion Stoilow" of the Romanian Academy, 21,
Calea Grivitei Street, 010702, Bucharest, Romania, AND\newline 
Max Planck Institut f\"ur Mathematik, Vivatsgasse 7, 53111 Bonn, Germany, AND\newline 
University of Bucharest, Research Center in Geometry, Topology and Algebra, Faculty of Mathematics and Informatics, 14 Academiei Str., Bucharest, Romania}
\email{lornea@fmi.unibuc.ro, liviu.ornea@imar.ro}
\email{aotiman@mpim-bonn.mpg.de, alexandra\_otiman@yahoo.com}

\begin{document}

\begin{abstract}
We give an equivalent definition of compact locally conformally hyperk\"ahler manifolds in terms of the existence of a nondegenerate complex two-form with natural properties. This is a conformal analogue of Beauville's stating that a compact K\"ahler manifold admitting a holomorphic symplectic form is hyperk\"ahler.\\

\end{abstract}

\maketitle
\section{Introduction}

A complex manifold $(M,J)$ is called {\em locally conformally K\"ahler} (LCK for short) if it admits a Hermitian metric $g$ such that the two form $\omega(X,Y):=g(JX,Y)$ satisfies the integrability condition
$$d\omega=\theta\wedge\omega$$
with respect to a {\em closed} one-form $\theta$, called the {\em Lee form}.

It is then immediate that locally, the metric $g$ is conformal to some local K\"ahler metrics $g'_U:=e^{-f_U}g\restrict{U}$, where $\theta\restrict{U}=df_U$.

An equivalent definition requires that the universal cover $(\tilde M, J)$ of $(M,J)$ admits a K\"ahler metric with respect to which the deck group acts by holomorphic homotheties, see \cite{do}. This K\"ahler metric on $\tilde M$, which is globally conformal with the pull-back of the LCK metric $g$, is in fact obtained by gluing the pulled-back local K\"ahler metrics $g'_U$. Note that, admitting homotheties, the universal cover of a LCK manifold is never compact.


There are  many examples of LCK manifolds: diagonal and non-diagonal Hopf manifolds, Kodaira surfaces,  Kato surfaces, some Oeljeklaus-Toma manifolds etc. All complex submanifolds of LCK manifolds are LCK. See \eg \cite{do}, \cite{ov_jgp} and the bibliography therein.

The LCK condition is conformally invariant: if $g$ is LCK with Lee form $\theta$, then $e^fg$  is LCK with Lee form $\theta+df$. One can then speak about an {\em LCK structure} on $(M,J)$ given by the couple $([g], [\theta])$. where $[,]$ denotes conformal class, respectively de Rham cohomology class. 

The complex structure $J$ is parallel with respect to the Weyl connection $D$ associated to $\theta$ and $[g]$ (acting by $Dg=\theta\otimes g$). This means that $D$ is, in fact, obtained by gluing the Levi-Civita connections of the local K\"ahler metrics $g'_U$ and therefore, the Levi-Civita connection of the K\"ahler metric on $\tilde M$ is the pull-back of the Weyl connection on $M$.

On a compact LCK manifold, if the local K\"ahler metrics are Einstein, a wellknown result by Gauduchon, \cite{gau_wein}, says that they are in fact Ricci flat. In this case,  the LCK metric itself is called {\em Einstein-Weyl} (see \cite{pps}, \cite{ov_ew}) and has the property that in the conformal class of $g$ there exists a metric with parallel Lee form, unique up to homotheties, known in the literature as Vaisman (see \cite{gau_wein}).

A particular example of Einstein-Weyl metric is {\em locally conformally hyperk\"ahler} (LCHK). In this case, $\dim_\R M=2n$ is divisible by 4, and $M$ admits a hyperhermitian 
structure $(I,J,K, g)$ such that all three Hermitian couples $(g,I)$, $(g,J)$, $(g,K)$ are LCK with respect to the same Lee form $\theta$. The K\"ahler metric of the universal cover has then holonomy included in $\Sp{\frac{n}{2}}$, thus being Calabi-Yau. See \eg \cite{cab_sw}, \cite{op}, \cite{pps}, \cite{V}. The quaternionic Hopf manifold is an important example. A complete list of compact, homogeneous LCHK manifolds is given in \cite{op}. 

One easily verifies that on compact LCHK manifolds, the 2-form $\Omega:=\omega_J+\1 \omega_K$ is nondegenerate and produces a volume form by $\Omega^{\frac{n}{2}} \wedge \overline{\Omega}^{\frac{n}{2}}=c \cdot d\mathrm{vol}_g$, for a positive constant $c$. 
Moreover, clearly $\Omega$ satisfies the equation $d\Omega=\theta \wedge \Omega$. The aim of this note is to prove that these conditions are also sufficient to define an LCHK structure:

\begin{theorem}\label{main} Let $(M, J, g)$ be a compact locally conformally K\"ahler manifold of real dimension $2n$ and $\theta$ the Lee form of $g$. Then $g$ is locally conformally hyperk\" ahler if and only if there exists a non-degenerate $(2, 0)$-form $\Omega$  such that 
$$d\Omega=\theta \wedge \Omega \qquad \text{and}\quad \Omega^{\frac{n}{2}} \wedge \overline{\Omega}^{\frac{n}{2}}=c \cdot d\mathrm{vol}_g,$$ where $c \in \mathbb{R}_{+}$, and $d\mathrm{vol}_g$ is the volume form of $g$.
\end{theorem}

\begin{remark} Note that the existence of a nondegenerate $(2, 0)$-form implies that the real dimension is a multiple of 4.
\end{remark}

\begin{remark}
A complex manifold admitting a nondegenerate $(2, 0)$-form $\omega$ such that a closed one-form $\theta \in \Lambda^1(M, \C)$ exists and $d \omega = \theta \wedge \omega$, is called {\em complex locally conformally symplectic} (CLCS). The Lee form can be real or complex. CLCS manifolds   first appeared in \cite[Section 5]{lic1}, motivated by the  examples of even-dimensional leaves of the natural generalized foliation of a complex Jacobi manifolds (recall that real LCS structures also appear as leaves of real Jacobi manifolds). 
\end{remark}

Theorem A can then be viewed as a conformal version of the celebrated Beauville theorem stating that a compact K\" ahler manifold, admitting a holomorphic symplectic form is hyperk\"ahler, \cite{beauville}. Our proof follows the ideas of Beauville's, but is different, the main difficulty being the fact that the universal cover of a LCK manifold is K\"ahler, but never  compact, and hence one has to make a long detour to use the Weitzenb\"ock formula on the compact LCK manifold $M$.  Moreover, there is no analogue of Yau's theorem on LCK manifolds or non-compact K\" ahler manifolds, which is an essential ingredient in Beauville's proof for obtaining a K\"ahler Ricci flat metric. Nevertheless, the rather strong condition $\Omega^{\frac{n}{2}} \wedge \overline{\Omega}^{\frac{n}{2}}=c \cdot d\mathrm{vol}_g$ is meant to replace Yau's theorem and produce eventually the K\" ahler Ricci flat metric on the universal cover of $M$.

A generalization of Beauville's theorem, but in a different sense, namely when the manifold is compact K\" ahler, but admitting a twisted holomorphic form, is presented in \cite{nico}.

\section{Proof of Theorem A}

The following two lemmas will be used in the proof.

\begin{lemma}\label{ricciflat}
Let $(M, J, g)$ be a compact LCK manifold with a non-degenerate $(2, 0)$-form $\Omega$ such that $d\Omega=\theta \wedge \Omega$ and $\Omega^{\frac{n}{2}} \wedge \overline{\Omega}^{\frac{n}{2}}=c \cdot d\mathrm{vol}_g$, where $\theta$ is the Lee form of $g$, $c \in \mathbb{R}_{+}$ and $d\mathrm{vol}_g$ is the volume form of $g$. Then $g$ is Einstein-Weyl.
\end{lemma}

\begin{proof}

Let $\tilde{M}$ be the universal cover of $M$, endowed with the complex structure $\tilde{J}=\pi^*J$, where $\pi: \tilde{M} \rightarrow M$. Let $\pi^*\theta=df$ and $\tilde{g}$ the K\"ahler metric given by $e^{-f}\pi^*g$. Denote by $\tilde{\Omega}:=e^{-f} \pi^*\Omega$. This is a $(2, 0)$-form which is closed, as a consequence of $d\Omega=\theta\wedge \Omega$, hence it is holomorphic.  Moreover, $\tilde{\Omega}^{\frac{n}{2}}\wedge \overline{\tilde{\Omega}^{\frac{n}{2}}}=c \cdot d\mathrm{vol}_{\tilde{g}}$, since $d\mathrm{vol}_{\tilde{g}}=e^{-nf}d\mathrm{vol}_g$.

Note that if instead of $g$ we consider the metric $g_1=e^f g$ with its corresponding Lee form $\theta_1=\theta+df$, then taking $\Omega_1=e^f \Omega$, we still obtain a non-degenerate form of type $(2, 0)$ satisfying $d\Omega_1=\theta_1 \wedge \Omega_1$ and $\Omega_1^{\frac{n}{2}} \wedge \overline{\Omega_1}^{\frac{n}{2}}=c \cdot d\mathrm{vol}_{g_1}$, therefore the statement of the lemma is conformally invariant.

Let $\tilde{K}$ be the canonical bundle of $\tilde{M}$. 
There is a natural Hermitian structure on $\tilde{K}$ which comes from $\tilde{g}$ given by $\alpha \wedge *\bar{\beta}=\tilde{g}(\alpha, \beta)\ d\mathrm{vol}_{\tilde{g}}$. Note that because $n$ is even, $*\overline{\beta}=\overline{\beta}$ (see \cite[Exercise 18.2.1]{moroianu}), thus, $\tilde{g}(\tilde{\Omega}^{\frac{n}{2}},\tilde{\Omega}^{\frac{n}{2}})=c$. The curvature form of the Chern connection associated to $\tilde{g}$ is on one hand $\mathrm{i}\,  \partial\overline{\partial}\log \det (\tilde{g}_{ij})$, where $\det \tilde{g}_{ij}=\det \tilde{g}(\frac{\partial}{\partial z_i},\frac{\partial}{\partial \overline{z_j}})$,  and on the other hand, it is $-\mathrm{i}\, \partial\overline{\partial} \log \tilde{g}(\tilde{\Omega}^{\frac{n}{2}},\tilde{\Omega}^{\frac{n}{2}})=0.$ Since $\mathrm{i}\,  \partial\overline{\partial}\log \det (\tilde{g}_{ij})$ is the local expression of the Ricci form $\rho(X, Y)= \mathrm{Ric}_{\tilde{g}}(\tilde{J}X, Y)$, we conclude that $\tilde{g}$ is Ricci flat and hence, $g$ is Einstein-Weyl. 
In particular, in the conformal class of $g$, there exists a Vaisman metric, unique up to homotheties.
 \end{proof}
\begin{lemma}\label{relWeyl}
Let $h$ be the Hermitian structure induced by $g$ on $\Lambda^{2,0}T^*_{\C}M$ (that is $h(\omega, \eta)=g(\omega, \overline{\eta})$). The Weyl connection $D$ on $M$ satisfies $D h=-2\theta \otimes h$.
\end{lemma}
\begin{proof}

This is because the Hermitian structure $\tilde{h}$ induced by $\tilde{g}$ on $\Lambda^{2, 0}_{\C}\tilde{M}$ is given by $e^{2f}\pi^*h$. Then $\pi^*D=\nabla^{\tilde{g}}$ (see \cite{do}) implies that $\pi^*D(e^{2f}\pi^*h)=0$, which yields $(\pi^*D)( \pi^*h)=-2\pi^*\theta \otimes \pi^*h$ and our relation follows.
\end{proof}

\hfill

We proceed with the proof of Theorem A.

In terms of holonomy, to say that $g$ is locally conformally hyperk\" ahler is the same as proving that the holonomy of $\tilde{g}$ on $\tilde{M}$ is contained in $\Sp{\frac{n}{2}}$, which is equivalent to $\tilde{g}$ being hyperk\" ahler. According to the holonomy principle (see \eg \cite[Page 758]{beauville}), proving that the holonomy of $\tilde{g}$ is in $Sp(\frac{n}{2})$ is equivalent to the existence of a complex structure $\tilde{J}$ on $\tilde{M}$, with respect to which $\tilde{g}$ is K\"ahler and a holomorphic two-form $\tilde{\Omega}$, parallel with respect to the Levi-Civita connection of $\tilde{g}$, $\nabla^{\tilde{g}}$. We are going to prove that these are $\tilde{J}$ and $\tilde{\Omega}$ from the proof of \ref{ricciflat} and we already saw that $\tilde{\Omega}$ is holomorphic.
The non trivial part is then to prove:

\begin{lemma} 
$\tilde{\Omega}$ is $\tilde g$-parallel. 
\end{lemma}

\begin{proof} Since $\nabla^{\tilde{g}}=\pi^*D$, $\nabla^{\tilde{g}}\tilde{\Omega}=0$ is equivalent to 
\begin{equation}\label{relatiasuprema}
D\Omega=\theta \otimes \Omega.
\end{equation}

This is a conformally invariant relation on $M$: for  any smooth $f$ on $M$ we have
\begin{equation*}
D(e^f\Omega)=(\theta + df) \otimes e^f\Omega. 
\end{equation*}
Using thus the freedom of choosing any metric in the conformal class, with the corresponding change of $\theta$ and $\Omega$, we choose the Vaisman metric, unique up to homotheties and without loss of generality we assume it is $g$. In this case, the Lee form is harmonic, has constant norm and moreover, we can choose the Vaisman metric with the Lee form of norm 1. We shall use these facts in the following computations.

We apply the Weitzenb\" ock formula on the holomorphic form $\tilde\Omega$. According to \cite{moroianu} (see Theorem 20.2 and the beginning of the proof of Theorem 20.5), as $\tilde{g}$ is Ricci flat (by  \ref{ricciflat}) the curvature term vanishes identically and the Weitzenb\" ock formula reduces to:
$$(\nabla^{\tilde{g}})^*\nabla^{\tilde{g}}\, \tilde{\Omega}=0.$$
 However, $\tilde{M}$ is not compact and we cannot deduce by integration that $\nabla^{\tilde{g}} \tilde{\Omega}=0$. 

For simplicity, from now on we write  $\nabla$ for $\nabla^{\tilde{g}}$.
By \cite[Lemma 20.1]{moroianu},
\begin{equation}\label{dual}
\sum_{i=1}^{2n}\nabla^*\nabla\tilde{\Omega}=\nabla_{\nabla_{f_i}f_i} \tilde{\Omega} - \nabla_{f_i}\nabla_{f_i}\tilde{\Omega},
\end{equation}
where $\{f_i\}$ is a local $\tilde{g}$ - orthonormal frame. We can choose $f_i=e^{\frac{f}{2}}\pi^* e_i$, where $\{e_i\}$ is a local $g$ - orthonormal frame  on $M$.  Then \eqref{dual} implies 
\begin{equation*}
\sum_{i=1}^{2n}\nabla^*\nabla\tilde{\Omega}=e^f(\nabla_{\nabla_{\pi^*e_i}\pi^*e_i} \tilde{\Omega} - \nabla_{\pi^*e_i}\nabla_{\pi^*e_i}\tilde{\Omega}),
\end{equation*}
and hence
\begin{equation*}
\sum_{i=1}^{2n}\nabla_{\nabla_{\pi^*e_i}\pi^*e_i} \tilde{\Omega} - \nabla_{\pi^*e_i}\nabla_{\pi^*e_i}\tilde{\Omega}=0.
\end{equation*}
Writing now $\tilde{\Omega}=e^{-f}\pi^*\Omega$, the above relation gives:
\begin{equation}\label{lung}
\begin{split}
0&=\sum_{i=1}^{2n}e^{-f}(\nabla_{\nabla_{\pi^*e_i}\pi^*e_i} \pi^*\Omega - \nabla_{\pi^*e_i}\nabla_{\pi^*e_i}\pi^*\Omega-\pi^*\theta(\nabla_{\pi^*e_i} \pi^*e_i)\pi^*\Omega \\
&+ 2\pi^*\theta(\pi^*e_i)\nabla_{\pi^*e_i}\pi^*\Omega 
 - (\pi^*\theta(\pi^*e_i))^2\pi^*\Omega 
 +\pi^*e_i(\pi^*\theta(\pi^*e_i))\pi^*\Omega).
 \end{split}
\end{equation}
As $\nabla=\pi^*D$, \eqref{lung} descends on $M$ to the following equality:
\begin{equation}\label{peM}
\sum_{i=1}^{2n}D_{D_{e_i}e_i}\Omega - D_{e_i}D_{e_i}\Omega-\theta(D_{e_i} e_i)\Omega + 2\theta(e_i)D_{e_i}\Omega - (\theta(e_i))^2\Omega +
 e_i(\theta(e_i))\Omega=0
\end{equation} 
We notice that $\sum_{i=1}^{2n}(\theta(e_i))^2\Omega=\Vert\theta\Vert^2_g\Omega = \Omega$. Recall (see \cite{do}) that 
\begin{equation}\label{weyl}
D = \nabla^g - \frac{1}{2}(\theta \otimes \mathrm{id}+ \mathrm{id} \otimes \theta - g \otimes \theta^{\sharp}).
\end{equation}
Using that $\theta$ is harmonic and \eqref{weyl}, we get:
\begin{equation*}
0=-\delta^g\theta = \sum_{i=1}^{2n}e_i(\theta(e_i)) - \theta(\nabla^g_{e_i}e_i)\\
=\sum_{i=1}^{2n}(e_i(\theta(e_i))- \theta(D_{e_i}e_i)) + n-1.
\end{equation*}
Consequently, \eqref{peM} rewrites as:
\begin{equation}\label{peMsimplificat}
\sum_{i=1}^{2n}D_{D_{e_i}e_i}\Omega - D_{e_i}D_{e_i}\Omega + 2\theta(e_i)D_{e_i}\Omega = n\Omega.
\end{equation}
The goal is to prove that $\nabla\tilde{\Omega}=0$, that is $D\Omega=\theta \otimes \Omega$. Hence, if we define $\cad:= D - \theta\, \otimes$, we need to show that $\cad \Omega=0$. If $\cad^*$ is the adjoint of $\cad$, as $M$ is compact, $\cad \Omega=0$ will follow from:
\begin{equation*}
\int_M h(\cad^*\cad\, \Omega, \Omega)d\mathrm{vol}_g = 0.
\end{equation*}
In order to find an explicit expression of $\cad^*$, we use the same method as in the proof of  \cite[Lemma 20.1]{moroianu}. Let $\eta \otimes \sigma \in \Gamma(\Lambda^1_\C \otimes \Lambda^{2, 0}_{\C})$ and $s \in \Omega^{2, 0}(M)$. Define the one-form $\alpha(X):=h(\eta(X)\sigma, s)$.
Then taking a $\nabla^g$-parallel local frame $\{e_i\}$, using \ref{relWeyl} and the fact that $\theta$ is real, we derive:
\begin{equation*}
\begin{split}
-\delta^g \alpha & = \sum_{i=1}^{2n}e_i(\alpha(e_i))= \sum_{i=1}^{2n}e_i(h(\eta(e_i)\sigma, s)) \\
& = \sum_{i=1}^{2n}(D_{e_i}h)(\eta(e_i)\sigma, s)+ h(D_{e_i}(\eta(e_i)\sigma), s) + h(\eta(e_i)\sigma, D_{e_i}s)\\
& = \sum_{i=1}^{2n}(-2 \theta(e_i)h(\eta(e_i)\sigma, s) + h(e_i(\eta(e_i))\sigma, s)+h(\eta(e_i)D_{e_i}\sigma, s)) + h(\eta \otimes \sigma, Ds)\\
& = h(D_{\eta^{\sharp}}\sigma -(\delta^g \eta)\sigma, s) + h(\eta \otimes \sigma, Ds-2\theta \otimes s).
\end{split}
\end{equation*}
After integration on $M$, this implies:
\begin{equation*}
\int_M h(\eta \otimes \sigma, Ds-2\theta \otimes s)d\mathrm{vol}_g= \int_M h(-D_{\eta^{\sharp}}\sigma +(\delta^g \eta)\sigma, s)d\mathrm{vol}_g
\end{equation*}
and hence, the adjoint of $D-2\theta \otimes$ acts as follows:
$$(D-2\theta \otimes)^*(\eta \otimes \sigma)=(\delta^g \eta)\sigma-D_{\eta^{\sharp}}\sigma.$$
We define 
$$T: \Gamma(\Lambda^{2,0}(M)) \rightarrow \Gamma(\Lambda^1_{\C}\otimes \Lambda^{2, 0}(M)), \quad \text{by}\quad T(s)=\theta \otimes s.$$
Then $\cad^*=(D-2\theta \otimes)^*+T^*$. It is easy to see that $T^*(\eta \otimes \sigma)=\eta(\theta^{\sharp})\sigma$. Therefore, 
\begin{equation}\label{adjoint}
\cad^*(\eta \otimes \sigma)= (\delta^g \eta)\sigma-D_{\eta^{\sharp}}\sigma + \eta(\theta^{\sharp})\sigma.
\end{equation}
Note that in the computation above we denoted  by $h$, too, the Hermitian structure on $\Lambda^1_{\C} \otimes \Lambda^{2,0}_{\C}M$ which is a product of the Hermitian structure induced by $g$ on $\Lambda^1_{\C}M$ and $h$ defined in \ref{relWeyl}.

We are ready now to compute $\cad^*\cad\, \Omega$. First:
\begin{equation*}
\cad^*\cad\, \Omega=\cad^*(D\Omega - \theta \otimes \Omega)=\cad^*D\Omega-\cad^*\theta\otimes\Omega.
\end{equation*}
Now \eqref{adjoint} yields:
\begin{equation}\label{adjoint2}
\cad^*\theta \otimes \Omega=(\delta^g \theta)\Omega-D_{\theta^{\sharp}}\Omega + \Omega = \Omega - D_{\theta^{\sharp}}\Omega.
\end{equation}
The first term is equal to:
\begin{equation*}
\cad^*D\, \Omega =\sum_{i=1}^{2n}\cad^*(e_i \otimes D_{e_i}\Omega)
=\sum_{i=1}^{2n}(\delta^g e_i)D_{e_i}\Omega-D_{e_i}D_{e_i}\Omega+ e_i(\theta^\sharp)D_{e_i}\Omega.
\end{equation*}
But  \eqref{weyl} implies:
$$\delta^g e_i=\sum_{k=1}^{2n}-e_k(e_i(e_k))+e_i(\nabla^g_{e_k} e_k)=\sum_{k=1}^{2n}g(D_{e_k}e_k, e_i)+(1-n)\theta(e_i),$$ 
and thus:
\begin{equation}\label{adjoint1}
\begin{split}
\cad^*D\Omega&=\sum_{i=1}^{2n}\left(\sum_{k=1}^{2n}g(D_{e_k}e_k, e_i)D_{e_i}\Omega + (1-n)\theta(e_i)D_{e_i}\Omega-D_{e_i}D_{e_i}\Omega\right) + D_{\theta^{\sharp}}\Omega\\
&=\sum_{i=1}^{2n}(D_{D_{e_i}e_i}\Omega-D_{e_i}D_{e_i}\Omega) + (2-n)D_{\theta^{\sharp}}\Omega
\end{split}
\end{equation}
Combining \eqref{adjoint1} and \eqref{adjoint2} we arrive at:
\begin{equation*}
\cad^*\cad \Omega = \sum_{i=1}^{2n}(D_{D_{e_i}e_i}\Omega -D_{e_i}D_{e_i}\Omega) + (3-n)D_{\theta^{\sharp}}\Omega - \Omega,
\end{equation*}
which, together with  \eqref{peMsimplificat}, leads to the final result:
\begin{equation}\label{d*d}
\cad^*\cad \Omega=(n-1)(\Omega-D_{\theta^{\sharp}}\Omega).
\end{equation} 
Integrating \eqref{d*d} on $M$ we find:
\begin{equation*}
\begin{split}
\int_M h(\cad\Omega, \cad\Omega)d\mathrm{vol}_g&=\int_M h(\cad^*\cad \Omega, \Omega)d\mathrm{vol}_g\\
&=(n-1)\left(\int_M h(\Omega, \Omega)d\mathrm{vol}_g-\int_Mh(D_{\theta^{\sharp}}\Omega, \Omega)d\mathrm{vol}_g\right).
\end{split}
\end{equation*} 
In particular, the above equality proves that  $\int_Mh(D_{\theta^{\sharp}}\Omega, \Omega)d\mathrm{vol}_g$ is a real number. Moreover,
\begin{equation*}
\int_M (D_{\theta^{\sharp}}h)(\Omega, \Omega)d\mathrm{vol}_g=\int_M \theta^{\sharp}(h(\Omega, \Omega))d\mathrm{vol}_g-\int_M h(D_{\theta^{\sharp}}\Omega, \Omega)d\mathrm{vol}_g-\int_M h(\Omega, D_{\theta^{\sharp}}\Omega)d\mathrm{vol}_g.
\end{equation*}
The first integral in the right hand side vanishes, since by Stokes formula and the fact that $\theta^{\sharp}$ is Killing (see \cite{do}):
\begin{equation*}
\int_M \theta^{\sharp}(h(\Omega, \Omega))d\mathrm{vol}_g = \int_M \mathcal{L}_{\theta^{\sharp}} (h(\Omega, \Omega)d\mathrm{vol}_g)-\int_M h(\Omega, \Omega) \mathcal{L}_{
\theta^{\sharp}} d\mathrm{vol}_g=0.
\end{equation*}
Using once more \ref{relWeyl} we derive:
\begin{equation*}
-2\int_M h(\Omega, \Omega)d\mathrm{vol}_g=-2\mathrm{Re}\int_M h(D_{\theta^{\sharp}}\Omega, \Omega)d\mathrm{vol}_g=-2\int_M h(D_{\theta^{\sharp}}\Omega, \Omega)d\mathrm{vol}_g,
\end{equation*}
which implies
$$\int_M h(\Omega, \Omega)d\mathrm{vol}_g-\int_Mh(D_{\theta^{\sharp}}\Omega, \Omega)d\mathrm{vol}_g=0,$$
 therefore 
 $$\int_M h(\cad\Omega, \cad\Omega)d\mathrm{vol}_g=0,$$ 
 and thus $\cad \Omega=0$, which completes the proof.
\end{proof}

\noindent{\bf Acknowledgment:} We thank Nicolina Istrati, Andrei Moroianu and Victor Vuletescu for carefully reading a first draft of the paper and for very useful comments and to Johannes Sch\" afer for insightful discussions that improved a previous version.


\end{document}